\documentclass[11 pt]{article}

\usepackage{latexsym, amsmath, amssymb, longtable, booktabs,amscd,microtype,booktabs,amsthm,fancyhdr,amsfonts}

\usepackage{etoolbox}
\usepackage[affil-it]{authblk}
\usepackage[square]{natbib}
\usepackage{enumerate}
\usepackage{tikz}
\usetikzlibrary{matrix,arrows,decorations.pathmorphing}
\usepackage{hyperref}
\bibpunct{[}{]}{,}{n}{}{;} 

\numberwithin{equation}{section}
\newcommand\A{\mathbb{A}}
\newcommand\C{\mathbb{C}}

\newcommand\Ak{\A_{K}}
\newcommand\Ik{\I_{K}}
\newcommand\Ck{\mathcal{C}_{K}^{1}}

\newcommand\Zeta{\zeta(f,\chi)}
\newcommand\Dx{d^{*}x}
\newcommand\Zetaa{\zeta(\hat{f},\check{\chi})}
\newcommand\I{\mathbb{I}}
\newcommand\sO{\mathcal{O}}
\newcommand\sC{\mathcal{C}}
\newcommand\sS{\mathcal{S}}

\newcommand{\R}{\mathbb{R}}

\newcommand\Hom{{\mathrm {Hom}}}

\newtheorem{thm}{Theorem}[subsection]

\newtheorem{prop}[thm]{Proposition}
\newtheorem{lemma}[thm]{Lemma}
\theoremstyle{definition}

\theoremstyle{remark}
\newtheorem{rmk}[thm]{\textbf{Remark}}
\newtheorem{eg}[thm]{\textbf{Example}}

\theoremstyle{definition}
\newtheorem{dfn}{Definition}[subsection]

\theoremstyle{remark}

\theoremstyle{remark}

\makeatletter
\def\imod#1{\allowbreak\mkern10mu({\operator@font mod}\,\,#1)}

\makeatother

\title{\textbf{Meromorphic Continuation Of Global Zeta Function For Number Fields}}
\author{SUBHAM DE}
\affil{Department of Mathematics}
\affil{Indian Institute of Technology, Delhi, India.}
\affil{Email: \textbf{mas227132@iitd.ac.in}}

\date{}

\begin{document}

\maketitle
\thispagestyle{empty}

\begin{abstract}
	In the paper, we shall establish the existence of a meromorphic continuation of the Global Zeta Function $\zeta(f,\chi)$ of a \textit{Global Number Field} $K$ and also deduce the functional equation for the same, using different properties of \textit{the id\`{e}le class group} $\mathcal{C}_{K}^{1}$ of a global field $K$ extensively defined using basic notions of \textit{Ad\`{e}les} ($\Ak$) and \textit{Id\`{e}les} ($\Ik$) of $K$, and also evaluating Fourier Transforms of functions $f$ on the space $\mathcal{S}(\A_{K})$ of \textit{Ad\`{e}lic Schwartz-Bruhat Functions}. A brief overview of most of the concepts required to prove our desired result have been provided to the readers in the earlier sections of the text.
	\texttt{\textbf{2020 MSC: }}\texttt{Primary  11-02, 11R04, 11R37, 11M06, 11F70 }.\\
	\hspace*{50pt}\texttt{Secondary 11M41, 28C10, 11R56 }.
\end{abstract}

\pagenumbering{arabic}

\newpage
\tableofcontents

\pagestyle{fancy}

\fancyhead[LO,RE]{\markright}
\lfoot[]{Subham De}
\rfoot[]{IIT Delhi, India}

\section{Ad\`{e}les and Id\`{e}les: A Brief Introduction}
First we introduce some important notions and terms which we shall use to describe the notion of \textit{Ad\`{e}les} and \textit{Id\`{e}les} later on explicitly.\\
\subsection{Restricted Direct Products}
Consider $I:=\{\nu\}$ to be any \textit{indexing set}, and $I_{\infty}$ be any fixed finite subset of $I$. Assuming that, we have a \textit{locally compact} group $G_{\nu}$ ( not neccessarily abelian ), corresponding to every index $\nu\in I$, and also suppose we obtain a \textit{compact open} (consequently closed also under the topology) group $H_{\nu}$, relative to every such index $\nu\notin I_{\infty}$ having a subgroup structure corresponding to the same index $\nu$. Thus, we have the following definition:	
\begin{dfn}\label{def1}
(Restricted Direct Product)  For any $\nu\notin I_{\infty}$, we define the \textit{restricted direct product} of the group $G_{\nu}$ with respect to the subgroup $H_{\nu}$ as:
\begin{eqnarray}\label{1}
	G:={\prod\limits_{\nu\in I}}' G_{\nu}=\{(x_{\nu})_{\nu}\mbox{ }|\mbox{ }x_{\nu}\in G_{\nu}\mbox{ with }x_{\nu}\in H_{\nu}\mbox{ for infinitely many }\nu\}
\end{eqnarray}

\end{dfn}
\begin{rmk}\label{rmk1}
	The topology on $G$ is defined by fixing a neighbourhood base for the identity element which consists of sets of the form $\prod\limits_{\nu}N_{\nu}$, $N_{\nu}$ being a neighbourhood of $1$ in $G_{\nu}$ such that, $N_{\nu}=H_{\nu}$ for infinitely many $\nu$. Important to mention that, this topology is not the same as the product topology.
\end{rmk}
The restricted direct product $G$ defined in \textit{Definition\eqref{def1}} satisfies some important prortant properties:
\begin{prop}\label{prop1}
	The following holds for a restricted direct product $G$ of $G_{\nu}$ with respect to $H_{\nu}$ defined in \textit{Definition\eqref{def1}}:\\
	\begin{enumerate}  [\textbullet]
		\item $G$ is \textit{locally compact}.\\
		\item Any subgroup $Y\subseteq G$ has a compact closure $\Leftrightarrow$ For some family $\{K_{\nu}\}_{\nu}$ of compact subsets such that, $K_{\nu}\subseteq G_{\nu}$, and $K_{\nu}=H_{\nu}$ for infinitely many indices $\nu$, we then have, $Y\subseteq\prod\limits_{\nu}K_{\nu}$.
	\end{enumerate}
\end{prop}

\subsection{Characters of Restricted Direct Products}
We start with the definition of the \textit{character }of a group:
\begin{dfn}\label{def2}
	(Characters of a group)  A \textit{character} $\chi$ of a topological group $G$ is a continuous homomorphism $G\longrightarrow \mathbb{C}^{*}$.
\end{dfn}
	A priori using the notions of $G$ defined in \textit{Definition\eqref{def1}}, we have the following result:
	\begin{lemma}\label{lemma1}
		For $\chi\in \Hom_{cont.}(G,\mathbb{C}^{*})$, $\chi$ is trivial on all but finitely many $H_{\nu}$. Precisely, $\chi(y_{\nu})=1$ for infinitely many $\nu$,  $\forall$  $y:=(y_{\nu})_{\nu}\in G$, and also,
		\begin{center}
		$\chi(y)=\prod\limits_{\nu}\chi(y_{\nu})$
		\end{center}
	\end{lemma}
Using the statement of the above lemma, we establish the following important result we shall use later on:
\begin{lemma}\label{lemma2}
	Suppose that, $\chi_{\nu}\in \Hom_{cont.}(G_{\nu},\mathbb{C}^{*})$ for every $\nu$, and, $\chi_{\nu}|_{H_{\nu}}=1$ for infinitely many $\nu$. Then, $\chi=\prod\limits_{\nu} \chi_{\nu}$, and, $\chi\in \Hom_{cont.}(G_{\nu},\mathbb{C}^{*})$.
\end{lemma}
Suppose we consider the \textit{Pontryagin Dual} of the restricted direct product $G$ defined as in \textit{Definition\eqref{def1}}. Consider the dual groups $\hat{G_{\nu}}$ of $G_{\nu}$. Then the following theorem gives us the following relation between $\hat{G}$ and $\hat{G_{\nu}}$ as:
\begin{thm}
	\begin{eqnarray}
	\hat{G}\cong {\prod\limits_{\nu}}'\hat{G_{\nu}}
	\end{eqnarray}
	Where, the restricted direct product ${\prod\limits_{\nu}}'\hat{G_{\nu}}$ is with respect to the subgroups $K(G_{\nu},H_{\nu})$, where,
	\begin{eqnarray}\label{3}
	K(G_{\nu},H_{\nu}):=\{\chi_{\nu}\in \Hom_{cont.}(G_{\nu},\mathbb{C}^{*})\mbox{ }|\mbox{ }\chi_{\nu}|
_{H_{\nu}}=1\}	
	\end{eqnarray}
\end{thm}
Next, we define measures on $G$ and $\hat{G}$.

\subsection{Measures on Restricted Direct Products and their Duals}
\begin{prop}\label{prop2}
	Assume $G={\prod\limits_{\nu}}'G_{\nu}$ to be the restricted direct product of locally compact groups $G_{\nu}$ with respect to the family of compact subgroups $H_{\nu}\subseteq G_{\nu}$ for every $\nu\notin I_{\infty}$, we define the \textit{(left) Haar Measure} on $G_{\nu}$ to be $dg_{\nu}$ and normalize it using the condition,
	\begin{center}
	$\int\limits_{H_{\nu}}dg_{\nu}=1\mbox{, for almost all }\nu\notin I_{\infty}.$
	\end{center}
	Then, there exists a \textit{Haar Measure }$dg$ on $G$ sucth that, the restriction $dg_{S}$ of $dg$ to the group,
	\begin{center}
	$G_{S}=\prod\limits_{\nu\in S}G_{\nu}\times \prod\limits_{\nu\notin S}H_{\nu}$
	\end{center}
	For every finite set $S(\supseteq I_{\infty})$ of indices, is exactly equal to the \textit{product measure} on $G_{S}$, and moreover, the \textit{Haar Measure }$dg$ on $G$ is unique.
\end{prop}	
By the virtue of the above theorem, we can conclude that, 
\begin{eqnarray}\label{2}
dg=\prod\limits_{\nu}dg_{\nu}\mbox{\hspace{20pt}is well-defined.}
\end{eqnarray}	
And, $dg$ is the \textit{(left) Haar Measure} on $G$. This is also called the measure induced by the \textit{factor measures}( precisely $dg_{\nu}$ for every $\nu$ ).\\\\\
Following result reflects some of the important properties of the left Haar Measure $dg$ on $G$:
\begin{prop}\label{prop3}
	The following holds true for the \textit{(left) haar Measure} on a restricted direct product $G$ of locally compact groups $G_{\nu}$ with respect to the compact subgroups $H_{\nu}$:
	\begin{enumerate}  [(i)]
		\item Suppose $f$ is integrable on $G$. Then,
		\begin{center}
		$\int\limits_{G}f(g)dg=lim_{S}\int\limits_{G_{S}}f(g_{s})dg_{s}$
		\end{center}
		$S$ being any finite set of indices containing $I_{\infty}$.\\
		If $f$ is assumed only to be continuous, then the above holds true provided that, the integral can take values at infinity.
		\item Given any finite set of indices $S_{0}(\supseteq I_{\infty})$ and all such $\nu$ such that, $Vol(H_{\nu},dg_{\nu})\neq 1$, suppose we have a continuous integrable function $f_{\nu}$ on $G_{\nu}$ corresponding to each such index $\nu$ such that,
		\begin{center}
		 $f_{\nu}|_{H_{\nu}}=1\mbox{, }\forall\mbox{  }\nu\notin S_{0}$
		 \end{center}
		 Then the relation $f(g)=\prod\limits_{\nu}f_{\nu}(g_{\nu})$,  $\forall$  $g:=(g_{\nu})_{\nu}\in G$ is well-defined and continuous on $G$.\\
		 Suppose $S$ be any finite set of indices ( $S$ can be $S_{0}$ ), then,
		 \begin{center}
		 $\int\limits_{G_{S}}f(g_{S})dg_{S}=\prod\limits_{\nu\in S}(\int\limits_{G_{\nu}}f_{\nu}(g_{\nu})dg_{\nu})$
		 \end{center}
		 More generally, 
		 \begin{center}
		 $\int\limits_{G}f(g)dg=\prod\limits_{\nu}(\int\limits_{G_{\nu}}f_{\nu}(g_{\nu})dg_{\nu})$
		 \end{center}
		 And, also $f\in L^{1}(G)$, provided, $\prod\limits_{\nu}(\int\limits_{G_{\nu}}f_{\nu}(g_{\nu})dg_{\nu})<\infty$.\label{4}
		 \item $\{f_{\nu}\}$ and $f$ be mentioned a priori, such that, moreover, $f_{\nu}$ is the \textit{characteristic function} for $H_{\nu}$ for almost all $\nu$. Then, $f$ is integrable, and, in case of abelian groups, the \textit{Fourier Transform} of $f$, i.e., $\hat{f}$ is likewise integrable and,
		 \begin{center}
		 $\hat{f_{\nu}}(g)=\prod\limits_{\nu}\hat{f_{\nu}}(g_{\nu}).$
		 \end{center}
	\end{enumerate}
\end{prop}
Now, Assuming the expression of $dg$ in \eqref{2}, we normalize each $dg_{\nu}$ for every such index $\nu$ such that, $Vol(H_{\nu})=1$ for almost all $\nu$. And, suppose,
\begin{center}
	$d\chi_{\nu}=\hat{dg_{\nu}}$,  for every $\nu$.
\end{center}
	Where, $d\chi_{\nu}$ denotes the \textit{dual measure} of $dg_{\nu}$ on $\hat{G_{\nu}}$ for every $\nu$.\\
	Hence for each $\nu$ and $f\in L^{1}(G_{\nu})$, we have, using description of \textit{Forurier Transform},
	\begin{center}
		$\hat{f_{\nu}}(\chi_{\nu})=\int\limits_{G_{\nu}}f_{\nu}(g_{\nu})\overline{\chi_{\nu}}(g_{\nu})dg_{\nu}.$
	\end{center}
	In case of $f_{\nu}$ being the \textit{characteristic function} of $H_{\nu}$, hence being integrable and of positive type on $G_{\nu}$ for every $\nu$, we deduce using \textit{orthogonality relations} that,
	\begin{eqnarray}
	\hat{f_{\nu}}(\chi_{\nu})=\int\limits_{H_{\nu}}\chi_{\nu}(g_{\nu})dg_{\nu}=\left\{
			\begin{array}{cc}
			Vol(H_{\nu})\mbox{ }, &\mbox{ if }\chi_{\nu}|_{H_{\nu}}=1,\\
			0\mbox{ }, &\mbox{ otherwise }.
			\end{array}
			\right. 
	\end{eqnarray}
Consequently, the group, 
\begin{eqnarray}
H_{\nu}^{*}=\{\chi_{\nu}\in \Hom_{cont.}(G_{\nu},\mathbb{C}^{*})\mbox{ }|\mbox{ \hspace*{5pt} }\chi_{\nu}|_{H_{\nu}}=1\}=K(G_{\nu},H_{\nu})\mbox{  },\mbox{using }\eqref{3}
\end{eqnarray}	
	is a subgroup of $\hat{G_{\nu}}$.\\
	Therefore, applying \textit{Proposition 1.3.2.\eqref{4}}, we thus obtain the \textit{Fourier Inversion Formula }stated as:
	\begin{thm}\label{thm1}
		(Fourier Inversion Formula)  We have, for every $f\in V^{1}(G)$ (Space of $L^{1}$ functions in $V(G)$, where $V(G)$ denotes the complex span of continuous functions on $G$ of positive type),\begin{eqnarray}
		f(g)=\int\limits_{\hat{G}}\hat{f}(\chi)\chi(g)d\chi 
		\end{eqnarray}
		Where, $\hat{G}$ is the \textit{Pontryagin Dual Group} of $G$ and $d\chi$ is the \textit{dual measure }on $\hat{G}$ to the Haar Measure $dg$ on $G$.
	\end{thm}
	Using the \textit{Fourier Inversion formula } and the \textit{Proposition\eqref{prop3}}, we conclude that,
	\begin{center}
		$Vol(H_{\nu}).Vol(H_{\nu}^{*})=1$
	\end{center}
Where $Vol(H_{\nu})$ is relative to the measure $dg_{\nu}$, and $Vol(H_{\nu}^{*})$ is relative to the measure $d\chi_{\nu}$; such that, $Vol(H_{\nu}^{*})=1$  for almost all $\nu$, and, $d\chi=\hat{dg}$ as mentioned earlier.\\\\

\subsection{Valuations on Number Fields}
\begin{dfn}\label{def3}
	(Absolute Value) An \textit{Absolute Value} of a field $K$ is a funtion, $|.|:K \longrightarrow\mathbb{R}$ satisfying,
		\begin{enumerate} [(i)]
			\item $|x|\geq 0$,  $\forall$ $x\in K$, and, $|x|=0$  $\Leftrightarrow$  $x=0$.
			\item $|xy|=|x||y|$,  $\forall$  $x,y\in K$.
			\item $|x+y|\leq |x|+|y|$,  $\forall$  $x,y\in K$  (\textit{triangle inequality}).
		\end{enumerate}
\end{dfn}

\begin{dfn}\label{def4}
	(Valuations) A \textit{valuation }of a field $K$ is a map, $v:K\longleftarrow \R\cup\infty$ satisfying,
	\begin{enumerate} [(i)]
		\item $v(x)=\infty$  $\Leftrightarrow$  $x=0$.
		\item $v(xy)=v(x)+v(y)$,  $\forall$  $x,y\in K$.
		\item $v(x+y)\geq min.\{v(x),v(y)\}$,  $\forall$  $x,y\in K$  
	\end{enumerate}
\end{dfn}
\begin{dfn}\label{def5}
	(Non-Archimedean Valuation) The valuation $v$ is defined to be \textit{non-Archimedean}, if it satisfies all the three conditions mentioned in the \textit{Definition\eqref{def4}}. Correspondingly, $|.|$ is defined to be the \textit{non-Archimedean} absolute value on $K$, if, $|n|$ is bounded for every $n\in \mathbb{N}$ on $K$. 
\end{dfn}

\begin{dfn}\label{def6}
	(Archimedean Valuation) The valuation $v$ is defined to be \textit{Archimedean}, if it does not satisfy the third condition, but satisfies the other two mentioned in the \textit{Definition\eqref{def4}}. Correspondingly, $|.|$ is defined to be the \textit{Archimedean} absolute value on $K$, if, $|n|$ is bounded for every $n\in \mathbb{N}$ on $K$.\\
	In other words, $v$ is \textit{non-Archimedean} if it is not \textit{Archimedean}. 
\end{dfn}
Clearly, from the above two definitions of \textit{valuations} and \textit{absolute value} on a field $K$, we can establish the following relation between them,
\begin{center}
	$v(x)=-log(|x|)$,  $\forall$  $x\in K$.
\end{center}
Where the $log$ can be considered with respect to some prime $p$ as the base, in case of \textit{p-adic valuations}.\\\\
\begin{eg}\label{eg1}
	Consider the \textit{p-adic valuation} and the \textit{p-adic absolute value} on $\mathbb{Q}$ and $\mathbb{Z}$, which is \textit{non-Archimedean}.
\end{eg}
	\begin{prop}\label{prop4}
		Every valuation on $\mathbb{Q}$ is either equivalent to the \textit{p-adic absolute value} $|.|_{p}$ or the \textit{usual absolute value} $|.|_{\infty}$.
	\end{prop}
Having defined all the tools neccessary, we finally introduce the notion of \textit{Ad\`{e}eles} and \textit{Id\`{e}les} of a global field $K$.

\subsection{Ad\`{e}les and Id\`{e}les}
\begin{dfn}\label{def7}
	(Valuation Ring) Given a global field $K$, suppose we denote $K_{\nu}$ as the completion of $K$ at a \textit{non-Archimedean} place $\nu$. Hence we define the \textit{Valuation Ring} of $K_{\nu}$ as,
	\begin{center}
		$\mathcal{O}_{\nu}:=\{x_{\nu}\in K_{\nu}\mbox{ }|\mbox{ }|x_{\nu}|_{\nu}\leq 1\}$	
	\end{center}
$|.|_{\nu}$ denoting the \textit{non-Archimedean absolute value} at a place $\nu$.
\end{dfn}
\begin{rmk}\label{rmk2}
	$(K_{\nu},+)$ is a \textit{locally compact} additive group.
	\begin{eg}\label{eg2}
		If $K$ be an \textit{algebraic number field}, then $K_{\nu}$ shall either be $\mathbb{R}$,$\mathbb{C}$ or, a \textit{p-adic field}.
	\end{eg}
\end{rmk}
For every finite place $\nu$, each such $K_{\nu}$ admits of a local ring of integers $\mathcal{O}_{\nu}$ as defined above; which is open and \textit{compact }as a subgroup. Hence we have our definition:
\begin{dfn}\label{def8}
	(Ad\`{e}le Group)  The \textit{Ad\`{e}le }Group $\A_{K}$ of a global field $K$ is defined to be the restricted direct product of $K_{\nu}$ over all $\nu$ with respect to the subgroups $\mathcal{O}_{\nu}$( $\nu$ is finite ), i.e.,
	\begin{eqnarray}
	\A_{K}:={\prod\limits_{\nu}}'K_{\nu}
	\end{eqnarray}
\end{dfn}
Again, for every finite place $\nu$, we consider the \textit{locally compact} multiplicative groups $(K_{\nu}^{*},.)$, then each such $K_{\nu}^{*}$ admits of a local ring of units $\mathcal{O}_{\nu}^{\times}$( $\nu$ is finite ) which is open and \textit{compact} as a subgroup. Thus we have our other definition:
\begin{dfn}\label{def9}
	(Id\`{e}le Group)  The \textit{Id\`{e}le Group} $\I_{K}$ of a global field $K$ is the group of units of the Ad\`{e}le Group $\A_{K}$ and precisely is defined to be the restricted direct product of $K_{\nu}^{*}$ over all $\nu$ with respect to the subgroups $\mathcal{O}_{\nu}^{\times}$( $\nu$ is finite ), i.e.,
	\begin{eqnarray}
	\I_{K}:={\prod\limits_{\nu}}'K_{\nu}^{*}
	\end{eqnarray}
\end{dfn}
\begin{rmk}\label{rmk3}
	From \textit{Definition\eqref{def8}}, it is evident that, there exists a well-defined algebraic embedding,
	\begin{center}
		$K\longrightarrow \A_{K}$\\
			\vspace*{10pt}
			$x\mapsto(x,x,x,......)$
	\end{center}
\end{rmk}
\begin{rmk}\label{rmk4}
	From \textit{Definition\eqref{def9}}, it is evident that, there exists a well-defined algebraic embedding,
	\begin{center}
		$K^{*}\longrightarrow \I_{K}$\\
			\vspace*{10pt}
			$x\mapsto(x,x,x,......)$
		\end{center}
	\end{rmk}
\begin{rmk}\label{rmk5}
	The \textit{Ad\`{e}le Group} $\A_{K}$ admits of a ring structure $(\A_{K},+,.)$, and , a priori, by \textit{Definition\eqref{def9}}, we have, as mentioned,
		\begin{center}
			$\I_{K}\cong \A_{K}^{\times}$
		\end{center}
	Although this is not a topological embedding, since the toplology induced by $\A_{K}$ is \textit{coarser} than the topology defined as the restricted direct product on $\I_{K}$.
\end{rmk}
\begin{rmk}\label{rmk6}
	The field $K$ as a group is discrete and \textit{cocompact} subgroup of $\A_{K}$.
\end{rmk}
In terms of \textit{valuations}, we can also give alternative definitions for \textit{Ad\`{e}les} and \textit{Id\`{e}les}.
\begin{dfn}\label{def10}
	(Ad\`{e}le Group)  Given a global field $K$, suppose for each finite place $\nu$, $K_{\nu}$ denotes the \textit{completion} of $K$ at the place $\nu$, having $\mathcal{O}_{\nu}$ as the local valuation ring at the finite non-Archimedean place $\nu$ defined as in \textit{Definition\eqref{def7}}. Then we define the \textit{Ad\`{e}le Group} of $K$ as,
	\begin{eqnarray}
	\A_{K}:={\prod\limits_{\nu}}'K_{\nu}=\{x:=(x_{\nu})_{\nu}\mbox{ }|\mbox{ }x_{\nu}\in K_{\nu}\mbox{ }\forall\mbox{ }\nu\mbox{ and, }x_{\nu}\in\mathcal{O}_{\nu}\mbox{ for infinitely many }\nu\}\\
	\hspace*{60pt}=\{x:=(x_{\nu})_{\nu}\mbox{ }|\mbox{ }x_{\nu}\in K_{\nu}\mbox{ }\forall\mbox{ }\nu\mbox{ and, }|x_{\nu}|_{\nu}\leq1\mbox{ for infinitely many }\nu\}
	\end{eqnarray}
\end{dfn}
\begin{dfn}\label{def11}
	(Id\`{e}le Group)  Given a global field $K$, suppose for each finite place $\nu$, $K_{\nu}$ denotes the \textit{completion} of $K$ at the place $\nu$, and $K_{\nu}^{*}$ denotes the locally compact multiplicative group having $\mathcal{O}_{\nu}^{\times}$ as the ring of units of the the local valuation ring $\mathcal{O}_{\nu}$ at the finite non-Archimedean place $\nu$ defined as in \textit{Definition\eqref{def7}}. Then we define the \textit{Id\`{e}le Group} of $K$ as,
	\begin{eqnarray}
	\I_{K}:={\prod\limits_{\nu}}'K_{\nu}^{*}=\{x:=(x_{\nu})_{\nu}\mbox{ }|\mbox{ }x_{\nu}\in K_{\nu}^{*}\mbox{ }\forall\mbox{ }\nu\mbox{ and, }x_{\nu}\in\mathcal{O}_{\nu}^{\times}\mbox{ for infinitely many }\nu\}\\
	\hspace*{60pt}=\{x:=(x_{\nu})_{\nu}\mbox{ }|\mbox{ }x_{\nu}\in K_{\nu}^{*}\mbox{ }\forall\mbox{ }\nu\mbox{ and, }|x_{\nu}|_{\nu}=1\mbox{ for infinitely many }\nu\}
	\end{eqnarray}
\end{dfn}
\begin{dfn}
	(Id\`{e}le Class Group)  Given a global field $K$, and $\I_{K}$ to be its \textit{Id\`{e}le Group}, we define the \textit{id\`{e}le class group} of $K$ as,
	\begin{eqnarray}
	\mathcal{C}_{K}:=\I_{K}/K^{*}
	\end{eqnarray}
\end{dfn}
Next, we shall introduce the notion of \textit{absolute value} on an Ad\`{e}le $\A_{K}$ and on an Id\`{e}le $\I_{K}$ of $K$.
\begin{dfn}\label{def12}
	Given a \textit{local field} $k$, we define the \textit{normalized absolute value} $|.|_{k}$ on $k$ as:
	\begin{enumerate}  [\textbullet]
		\item $|.|_{k}=|.|_{\infty}$, i.e., the usual absolute value if $k=\mathbb{R}$.
		\item $|z|_{k}=z.\overline{z}$, $\forall$ $z\in k$  if $k=\mathbb{C}$.
		\item For a non-Archimedean local field $K$ with uniformizing papameter $\pi$, we have,
		\begin{center}
			$|\pi|_{k}=\frac{1}{q}$,  where, $q=|\mathcal{O}_{k}/ \pi. \mathcal{O}_{k}|$	
		\end{center}
\end{enumerate}
\end{dfn}
\begin{dfn}
	Given a global field $K$, we know that, for every finite place $\nu$, $K_{\nu}$ is a local field which is also the \textit{completion} of $K$ at the place $\nu$. hence we define the \textit{normalized absolute value}, $|.|_{\A_{K}}:\A_{K}\longrightarrow \R_{+}^{*}$ in terms of the normalized absolute values $|.|_{\nu}$ on the completions $K_{\nu}$ as:
	\begin{eqnarray}\label{5}
		|x|_{\A_{K}}:=\prod\limits_{\nu}|x_{\nu}|_{\nu}, \mbox{  }   \forall \mbox{  }x:=(x_{\nu})_{\nu}\in \A_{K}
	\end{eqnarray}
\end{dfn}
\begin{rmk}\label{rmk7}
	for every $x\in \I_{K}$, we have, $|x|_{\A_{K}}=1$.
\end{rmk}
\begin{rmk}\label{rmk8}
	Using the above definitions, we can conclude that, $\mathcal{C}_{K}$ is not coompact with respect this absolute value, although a priori, we have that, $\A_{K}/K$ is compact.
\end{rmk}
\begin{dfn}\label{def13}
	(Id\`{e}le Class Group of Norm $1$)  Suppose $K$ be an \textit{algebraic number field} or a \textit{finitely generated function field} in one variable over a finite field $\mathbb{F}_{q}$, where $q$ is some power of prime. Then we define the \textit{norm $1$ Id\`{e}le Group} of $K$ as,
	\begin{center}
		$\I_{K}^{1}:=ker(|.|_{\A_{K}})$, \hspace*{20pt} where, the map $|.|_{\A_{K}}$ is defined in \eqref{5}.
	\end{center}
And, consequently, we give our definition of the \textit{Id\`{e}le Class Group of norm $1$} as,
\begin{eqnarray}\label{6}
\mathcal{C}_{K}^{1}:=\I_{K}^{1}/K^{*}
\end{eqnarray}
\end{dfn}
\begin{rmk}\label{rmk9}
	$\mathcal{C}_{K}^{1}$ is well-defined by the fact that, $K^{*}\hookrightarrow \I_{K}^{1}$ (Using \textit{Artin's Product Formula})

\end{rmk}
\begin{rmk}\label{rmk10}
	For any globbal field $K$, $\mathcal{C}_{K}^{1}$ is \textit{compact}.
\end{rmk}
Now that we have got the idea of the \textit{Ad\`{e}le groups} $\Ak$ and \textit{Id\`{e}le Groups} $\Ik$ for a global field $K$, we intend to explicitly perform \textit{Fourier Transforms} on a particular space $\mathcal{S}(\Ak)$ of all the \textit{Ad\`{e}lic Schwartz-Bruhat Functions} $f$.

\subsection{Fourier Transforms on $\mathcal{S}(\Ak)$}

Given a global field $K$ and its completion $K_{\nu}$ at aplace $\nu$ and $\sO_{\nu}$ being the local ring of integers of $K_{\nu}$ at $\nu$ for every finite place $\nu$; we can define corresponding \textit{Schwartz-Bruhat Spaces} $\sS(K_{\nu})$ of functions for every $\nu$.\\
Using \textit{Definition\eqref{def8}} , we define the space $\sS(\Ak)$ of \textit{Ad\`{e}lic Schwartz-Bruhat Functions} as:
\begin{center}
	$\sS(\Ak):=\bigotimes'\sS(K_{\nu})$
\end{center}
Where $\bigotimes'$ denotes the restricted tensor product of the indivudual Schwartz-Bruhat Spaces $\sS(K_{\nu})$ such that, $\sS(\Ak)$ has the following structure:
\begin{eqnarray}
\sS(\Ak):=\{f\in \otimes f_{\nu}\mbox{ }|\mbox{ }f_{\nu}\in \sS(K_{\nu})\mbox{ }\forall \mbox{ }\nu \mbox{, }f_{\nu}|_{\sO_{\nu}}=1\mbox{ for almost all }\nu \}
\end{eqnarray}
Therefore, $f\in \sS(\Ak)$  $\Rightarrow$  $f(x)=\prod\limits_{\nu}f_{\nu}(x_{\nu})$,  $\forall$  $x:=(x_{\nu})_{\nu}\in \Ak$\\
And $f$ is termed as the \textit{Ad\`{e}lic Schwartz-Bruhat Function}.
\begin{rmk}\label{rmk11}
	$\sS(\Ak)$ is \textit{dense} in $L^{2}(\Ak)$, where $L^{2}(\Ak)$ is defined with respect to the \textit{Haar Measure} on $\Ak$.
\end{rmk}
\begin{dfn}\label{def14}
	(Fourier Transformation Formula)  Given any $f\in \sS(\Ak)$, fixing a non-trivial continuous unitary character $\psi$ on $\Ak$, we define the \textit{Ad\`{e}lic Fourier Transform} of $f$ as:
	\begin{eqnarray}
	\hat{f}(y):=\int\limits_{\Ak}f(x)\psi(xy)dx
	\end{eqnarray}
\end{dfn}
Where $dx$ denotes the \textit{Haar Measure} on $\Ak$ normalized by the \textit{self-dual measure} for $\psi$.
\begin{rmk}\label{rmk12}
	The map, $f\mapsto \hat{f}$ defines an automorphism on $\sS(\Ak)$ which extends to an isometry on $L^{2}(\Ak)$.
\end{rmk}

\section{Global Zeta Function $\Zeta$}
\begin{dfn}\label{def15}
Given any $\C^{*}$-valued character $\chi$ if the \textit{Id\`{e}le Class Group} $\I_{K}$ such that, $\chi|_{K^{*}}=1$; i.e., in other words, $\chi$ being a \textit{quasi-character} of exponenet greater than $1$; i.e., in other words, an \textit{Id\`{e}le Class Character}; and for every $f\in \sS(\Ak)$, we define the \textit{Global Zeta Function} of the field $K$ as:
	\begin{eqnarray}
	\Zeta:=\int\limits_{\Ik}f(x)\chi(x)\Dx
	\end{eqnarray}

Where, $\Dx$ denotes the Haar Measure on $\Ik$, induced by the product measure $\prod\limits_{\nu}\Dx_{\nu}$ on $\prod\limits_{\nu}K_{\nu}^{*}$, for each non-Archimedean place $\nu$.
\end{dfn}
\begin{rmk}\label{rmk13}
Note that, here, $\Dx$ denotes the Haar Measure on $K_{\nu}^{*}$ for every non-Archimedean place $\nu$, usually having the following representation:
\begin{center}
	$\Dx_{\nu}:=c_{\nu}\frac{dx_{\nu}}{|x_{\nu}|_{\nu}}$
\end{center}
Where, $c:=(c_{\nu})_{\nu}$ is some constant factor that is usually introduced in order to normalize $\Dx$.\\\\
$c:=(c_{\nu})_{\nu}$ can be evaluated as:
\begin{center}
$c_{\nu}=\frac{q_{\nu}}{q_{\nu}-1}$
\end{center}
\vspace{10pt}
Where, $q_{\nu}:=\mathcal{N}(\nu)=q^{deg(\nu)}$,  where,  $deg(\nu):=[\mathbb{F}_{q_{\nu}}:\mathbb{F}_{q}]$, at every finite place $\nu$.\\
Such that, the \textit{Haar Measure} on $\sO_{\nu}^{\times}$  ($\sO_{\nu}$ being the local ring of integers at every non-Archimedean place $\nu$) shall be$=\sqrt{(\mathcal{N}(\mathcal{D}_{\nu})^{-1}}=\sqrt{q^{-d_{\nu}}}$, $q$ be a prime.\\\\
	Where, $\mathcal{D}_{\nu}$ denotes the \textit{local different} at every non-Archimedean place $\nu$. In fact, $d_{\nu}=0$ for all but finitely many non-Archimedean places $\nu$.
	\end{rmk}
	In our main theorem explicitly stated and proved later, we shall establish the fact that, $\Zeta$ is \textit{normally convergent} for $\sigma=Re(s)>1$,  and defines a holomorphic function in the region of its convergence, using rigorously, the fact that, $\chi$ has a representation,  $\chi=\mu|.|^{s}$, where, $\mu$ is a \textit{unitary character} of the Id\`{e}le Group $\Ik$.
\section{Riemann-Roch Theorem}
In this section, we shall state and prove one of the most unique and important theorems in the field of \textit{Harmonic Analysis} on \textit{Ad\`{e}lic Groups} which shall use to prove our main theorem. First, let us mention the \textit{Poisson Summation Formula} in order to prove the \textit{Riemann-Roch Theorem} .
\begin{thm}\label{thm2}
	(Poisson Summation Formula)  Consider $f\in \sS(\Ak)$, i.e., $f$ satisfies the following conditions,
	\begin{enumerate}  [(i)]
		\item $f\in L^{1}(\Ak)$,  and, $f$ is continuous.
		\item  $\sum\limits_{\gamma\in K}f(z(y+\gamma))$ converges for all id\`{e}les $z\in \Ik$ and for every ad\`{e}le $y\in \Ak$,  uniformly for $y$.
		\item  $\sum\limits_{\gamma\in K} |\hat{f}(z\gamma)|$ is convergent for every id\`{e}le $z\in \Ik$.
	\end{enumerate}
Then,
\begin{center}
	$\tilde{f}=\tilde{\hat{f}}$  
	\end{center}
Where, $\hat{f}$ denotes the \textit{Ad\`{e}lic Fourier Transform }of $f$, and, $\tilde{f}(x):=\sum\limits_{\gamma\in K} f(\gamma+x)$,  $\forall$  $x\in \Ak$.
In other words, 
\begin{eqnarray}
\sum\limits_{\gamma\in K}f(\gamma+x)=\sum\limits_{\gamma\in K}\hat{f}(\gamma+x).
\end{eqnarray}
\end{thm}
\begin{proof}
	We consider any function $\varphi$ on $\Ak/K$, induced by some $K$-invariant function $\varphi$ on $\Ak$. Therefore, by \textit{definition\eqref{def14}} of \textit{Fourier Transform}, we have,
	\begin{center}
		$\hat{\varphi}(z)=\int\limits_{\Ak/K}\varphi(t)\psi(tz)\overline{dt}$,  $\forall$  $z\in K$
	\end{center}
	Where, $\psi$ is a non-trivial continuous unitary character on $\Ak$; and, $\overline{dt}$ is the quotient measure on $\Ak/K$ induced by the measure $dt$ on $\Ak$; satisfying the relation:
	\begin{center}
	$	\int\limits_{\Ak/K}\hat{f}(t)\overline{dt}=\int\limits_{\Ak/K}\sum\limits_{\nu\in K}f(\nu+t)\overline{dt}=\int\limits_{\Ak}f(t)dt$
	\end{center}
	for every $f\in \sS(\Ak)$ such that, $f$ satisfies certain convergence properties.\\\\
	It is important to mention that, each of the integrals in the above  identity is well-defined, since, in the first two integrals, the integration variable $t$ assumes value from the quotient group $\Ak/K$.\\\\
	In order to prove the theorem, we shall need to apply the following two important lemmas:
	\begin{lemma}\label{lemma3}
		For every continuous function $f\in \sS(\Ak)$,
		\begin{center}
			$\hat{f}|_{K}=\hat{\tilde{f}}|_{K}$
		\end{center}
	\end{lemma}
\begin{proof}
	Using definition of \textit{Fourier Transform}, we obtain,
	\begin{center}
		$\hat{\tilde{f}}(z)=\int\limits_{\Ak/K}\tilde{f}(t)\psi(tz)\overline{dt}=\int\limits_{\Ak/K}(\sum\limits_{\gamma\in K}f(\gamma+t))\psi(tz)\overline{dt}$,  (Expanding $\tilde{f}$)  $\forall$  $z\in K$
	\end{center}
\begin{center}
	$=\int\limits_{\Ak/K}(\sum\limits_{\gamma\in K}f(\gamma+t)\psi((\gamma+t)z))\overline{dt}$
\end{center}
	[ Since $\psi$ being assumed to be unitary on $\Ak$, hence, $\psi|_{K}=1$ $\Rightarrow$ $\psi(tz)=\psi((\gamma+t)z)$,  $\forall$  $\gamma\in K$,  hence the equality holds by the definition of quotient measure on $\Ak$ relative to the counting measure on $K$.]
\begin{center}
	$=\int\limits_{\Ak}f(t)\psi(tz)dt$\\
	\vspace*{10pt}
	$=\hat{f}(z)$.  \hspace*{20pt}  $\forall$  $z\in K$
\end{center}
\end{proof}
\begin{lemma}\label{lemma4}
	For any $f\in \sS(\Ak)$ and for every $z\in K$,
	\begin{center}
		$\tilde{f}(z)=\sum\limits_{\gamma\in K}\hat{\tilde{f}}(\gamma)\overline{\psi}(\gamma z)$
	\end{center}
Where, $\overline{\psi}$ denotes the complex conjugate of $\psi$.
\end{lemma}
\begin{proof}
	We have, a priori by \textit{Lemma\eqref{lemma3}},
		\begin{center}
		$\hat{f}|_{K}=\hat{\tilde{f}}|_{K}$,  for $f\in \sS(\Ak)$
	\end{center}
Therefore, by the fact that, the sum, $\sum\limits_{\gamma\in K}\hat{f}(\gamma)\overline{\psi}(\gamma z)$ is normally convergent, we can assert that, the sum, $\sum\limits_{\gamma\in K}\hat{\tilde{f}}(\gamma)\overline{\psi}(\gamma z)$ is also normally convergent, i.e., precisely,
\begin{center}
	$\sum\limits_{\gamma\in K}|\hat{\tilde{f}}(\gamma)|<\infty$   \hspace*{20pt}[Since, $\psi$ is unitary, hence, $\overline{\psi}$ is also unitary.]
\end{center}
Since we have the counting measure on $K$, also also due to the fact that, the \textit{Pontryagin Dual} of $\Ak/K$ is $K$ itself under the discrete topology, hence using \textit{Fourier Inversion Formula}, we obtain,
\begin{center}
	$\tilde{f}(z)=\sum\limits_{\gamma\in K}\hat{\tilde{f}}(\gamma)\overline{\psi}(\gamma z)$,  $\forall$  $z\in K$.
\end{center}
And our claim is established.
\end{proof}

	Putting $z=0$ in \textit{Lemma\eqref{lemma4}}, we have,
	\begin{center}
		$\tilde{f}(0)=\sum\limits_{\gamma\in K}\hat{\tilde{f}}(\gamma)=\sum\limits_{\gamma\in K}\hat{f}(\gamma)$
	\end{center}
Although, 
\begin{center}
	$\tilde{f}(0)=\sum\limits_{\gamma\in K}f(\gamma)$  \hspace*{20pt}[Using definition of $\tilde{f}$]
\end{center}
And, hence, \begin{center}
	$\sum\limits_{\gamma\in K}f(\gamma)=\sum\limits_{\gamma\in K}\hat{f}(\gamma)$  \\
	\vspace*{10pt}
	$\Rightarrow\tilde{f}=\tilde{\hat{f}}$.  \hspace*{20pt}[Since, $f$ is $K$-invariant on $\Ak$]
\end{center}
And the \textit{Poisson Summation formula} is established.

\end{proof}
\begin{rmk}\label{rmk14}
	The sum, $\sum\limits_{\gamma\in K}f(\gamma x)$  for every $x\in \Ak$ is defined as the \textit{Average} for an Id\`{e}le $x$ in $\Ak$.
\end{rmk}
Next, we shall introduce the statement of the \textit{Riemann-Roch Theorem} which goes as follows:
\begin{thm}
	(Riemann-Roch Theorem)  Suppose $f\in \sS(\Ak)$, i.e., $f$ satisfies the following conditions:
	\begin{enumerate}  [(i)]
		\item $f\in L^{1}(\Ak)$,  and, $f$ is continuous.
		\item  $\sum\limits_{\gamma\in K}f(x+\gamma)$ converges for all ad\`{e}les $x\in \Ak$,  uniformly .
		\item  $\sum\limits_{\gamma\in K} |\hat{f}(\gamma)|$ is convergent .
	\end{enumerate}
Then,
\begin{eqnarray}
\sum\limits_{\gamma\in K}f(\gamma x)=\frac{1}{|x|}\sum\limits_{\gamma\in K}\hat{f}(\gamma x^{-1})
\end{eqnarray}
\end{thm}
\begin{proof}
	Fix, $x\in \Ik$. Now, we define $h\in \sS(\Ak)$ by,
	\begin{center}
		$h(y):=f(yx)$,  $\forall$  $y\in \Ak$.
	\end{center}
Then, 
\begin{eqnarray}\label{7}
	\sum\limits_{\gamma\in K}h(\gamma)=\sum\limits_{\gamma\in K} \hat{h}(\gamma) \hspace{20pt}[Applying \textit{Poisson Summation Formula}]
\end{eqnarray}
Although, by definition of \textit{Fourier Transform},
\begin{center}
	$\hat{h}(\gamma)=\int\limits_{\Ak}f(tx)\psi(t\gamma )dt$
\end{center}
\begin{center}
	$=\frac{1}{|x|}\int\limits_{\Ak}f(w)\psi(\gamma w x^{-1})$,   \hspace*{20pt}[Substituting, $w=tx$ in the integral]
\end{center}
\begin{center}
	$=\frac{1}{|x|}\hat{f}(\gamma x^{-1})$\hspace*{20pt} [By \textit{Definition\eqref{def14}}]
\end{center}
Applying \textit{Poisson Summation Formula( Theorem\eqref{thm2}}, and from \eqref{7}, we get,
	\begin{center}
	$\sum\limits_{\gamma\in K}f(\gamma x)=\frac{1}{|x|}\sum\limits_{\gamma\in K}\hat{f}(\gamma x^{-1})$
	\end{center}
And our theorem is established.
\end{proof}
\section{The Main Theorem}

\begin{thm}\label{thm3}
	The following holds true for the \textit{Global Zeta Function} $\Zeta$ of a number field $K$:
	\begin{enumerate}
		\item $\Zeta$ has a \textit{meromorphic extension }on $\C$.
		\item The \textit{extended global zeta function} $\Zeta$ is \textit{holomorphic }everywhere except when, $\mu=|.|^{-iy}$, $y\in \R$; hence having \textit{simple poles} at the points $s=iy$ and $s=1+iy$ with \textit{residues} given by,
		\begin{center}
			$-\kappa f(0)$,  and,  $\kappa \hat{f}(0)$  
		\end{center}
		respectively. Here, we can deduce that, $\kappa:=Vol.(\sC_{K}^{1})=$Volume of the \textit{Id\`{e}le Class Group of $K$ of norm $1$}.
		\item $\Zeta$ satisfies the functional equation,
		\begin{eqnarray}
		\Zeta=\Zetaa
		\end{eqnarray}
		Where $\hat{f}$ denotes the \textit{Fourier Transform} of a function $f\in \sS(\Ak)$ and, $\check{\chi}:=\chi^{-1}|.|$ is termed as the \textit{Shifted Dual} of the character $\chi$.
	\end{enumerate}
\end{thm}

\begin{proof}
	\begin{enumerate}
		\item For a number field $K$, a priori, we may write,
	\begin{eqnarray}\label{8}
	\Zeta=\int\limits_{0}^{\infty}\zeta_{t}(f,\chi)\frac{1}{t}dt
	\end{eqnarray}
	for every \textit{quasi-character} $\chi$ of exponent greater than $1$. Where, we define,
	\begin{eqnarray}\label{9}
	\zeta_{t}(f,\chi):=\int\limits_{\Ik}f(tx)\chi(tx)\Dx
		\end{eqnarray}
		Important to mention that, for any $x:=(x_{\nu})_{\nu}\in \Ak$, $t\in \R$,  we define the element,
		\begin{center}
			$tx:=(x_{\nu}')_{\nu};$  where, $x_{\nu}':=
			\left\{
			\begin{array}{cc}
			x_{\nu}$ $, &\mbox{ if }\nu\neq \nu',\\
			tx_{\nu}$ $, &\mbox{ if }\nu= \nu'.
			\end{array}
			\right. $	
		\end{center}

	For some specific non-Archimedean place $\nu'$.\\\\
Our aim is to first establish the \textit{Functional Equation} for $\zeta_{t}(f,\chi)$ using the \textit{Riemann-Roch Theorem} proved earlier, which leads us to prove the following proposition:
\begin{prop}\label{prop5}
	The function $\zeta_{t}(f,\chi)$ satisfies the functional equation,
	\begin{eqnarray}
		\zeta_{t}(f,\chi)=\zeta_{t^{-1}}(\hat{f},\check{\chi})+{\hat{f}(0)} \int\limits_{\Ck} \check{\chi}(\frac{x}{t})\Dx - f(0) \int\limits_{\Ck}\chi(tx)\Dx
			\end{eqnarray}
		\end{prop}
		\begin{proof}
			Using the \textit{Definition\eqref{def13}} of $\Ck$,
			\begin{center}
				$\zeta_{t}(f,\chi)=\int\limits_{\Ik^{1}}f(tx)\chi(tx)\Dx=\int\limits_{\Ck}(\sum\limits_{a\in K^{*}}f(atx))\chi(tx)\Dx=\int\limits_{\Ck}\chi(tx)\Dx(\sum\limits_{a\in K^{*}}f(atx))$  \hspace{100pt}[Since $\chi|_{K^{*}}=1$, by the hypothesis ]
			\end{center}
			therefore, we get, 
			\begin{center}
				$\zeta_{t}(f,\chi)+f(0)\int\limits_{\Ck}\chi(tx)\Dx=\int\limits_{\Ck}\chi(tx)\Dx(\sum\limits_{a\in K}f(atx))$
			\end{center}
			And, now, we apply the \textit{Riemann-Roch Theorem}, mentioned a priori, on the summand above so that, the Right Hand Side above yields the expression,
			\begin{center}
				$\int\limits_{\Ck}\chi(tx)\Dx(\sum\limits_{a\in K}f(atx))=\int\limits_{\Ck}\chi(tx)\Dx(\frac{1}{|tx|}\sum\limits_{a\in K}f(atx))$
				
			\end{center}
		\begin{center}
			$=\int\limits_{\Ck}\frac{\chi(tx)}{|tx|}\Dx(\sum\limits_{a\in K}\hat{f}(at^{-1}x^{-1}))$
		\end{center}
	\begin{center}
		$=\int\limits_{\Ck} |t^{-1} x| \chi(t x^{-1})\Dx (\sum\limits_{a\in K}\hat{f}(a t^{-1}x))$ \hspace*{20pt} [Substituting $x$ by $x^{-1}$]
	\end{center}
\begin{center}
	$=\zeta_{t^{-1}}(\hat{f},\check{\chi})+\hat{f}(0)\int\limits_{\Ck}\check{\chi}(\frac{x}{t})\Dx$\hspace*{20pt}[A priori, from \eqref{9}, substituting $t,f,\chi$ by, $t^{-1},\hat{f},\check{\chi}$ respectively.]
\end{center}
\begin{center}
	$	\zeta_{t}(f,\chi)=\zeta_{t^{-1}}(\hat{f},\check{\chi})+{\hat{f}(0)} \int\limits_{\Ck} \check{\chi}(\frac{x}{t})\Dx - f(0) \int\limits_{\Ck}\chi(tx)\Dx$
\end{center}
And the result is established.
\end{proof}
Using the above proposition, we shall prove our main theorem.
\begin{proof}
	Using \eqref{8}, as obtained from the definition of $\Zeta$, and applying properties of integration, we obtain that,
	\begin{center}\label{10}
	$	\Zeta=\int\limits_{0}^{1}\zeta_{t}(f,\chi)\frac{1}{t}dt+\int\limits_{1}^{\infty}\zeta_{t}(f,\chi)\frac{1}{t}dt$
	\end{center}
\begin{eqnarray}
	=:I_{1}+I_{2}\hspace{20pt}\mbox{(say)}
\end{eqnarray}
Where, 
\begin{eqnarray}\label{11}
I_{1}:=\int\limits_{0}^{1}\zeta_{t}(f,\chi)\frac{1}{t}dt
\end{eqnarray}
And,
\begin{eqnarray}\label{12}
I_{2}:=\int\limits_{1}^{\infty}\zeta_{t}(f,\chi)\frac{1}{t}dt
\end{eqnarray}
Now, 
\begin{center}\label{13}
$I_{2}:=\int\limits_{1}^{\infty}\zeta_{t}(f,\chi)\frac{1}{t}dt=\int\limits_{\{x\in\Ik\mbox{ }|\mbox{ }|x|\geq 1\}}f(x)\chi(x)\Dx$\hspace*{20pt}[Using Definition of $\zeta_{t}(f,\chi)$]
\end{center}
Which is \textit{normally convergent} for all $s\in \C$ .[Since, the integral above on the R.H.S. is convergent for $\sigma=Re(s)>1$]\\\\
Therefore, the integral $I_{2}$ is convergent for all $s\in \C$.\\\\
Using the \textit{functional equation} deduced in the \textit{Proposition\eqref{prop5}}, for $\zeta_{t}(f,\chi)$, we obtain,
\begin{eqnarray}\label{14}
	I_{1}:=\int\limits_{0}^{1}\zeta_{t}(f,\chi)\frac{1}{t}dt=\int\limits_{0}^{1}\zeta_{t^{-1}}(\hat{f},\check{\chi})\frac{1}{t}dt+\mathcal{E}
\end{eqnarray}
Where the error term $\mathcal{E}$ is defined as,
\begin{eqnarray}\label{15}
\mathcal{E}:=\int\limits_{0}^{1}\{\hat{f}(0) \int\limits_{\Ck} \check{\chi}(\frac{x}{t})\Dx - f(0) \int\limits_{\Ck}\chi(tx)\Dx\}\frac{1}{t}dt
\\
=\int\limits_{0}^{1}\{\hat{f}(0)\check{\chi}(t^{-1}) \int\limits_{\Ck} \check{\chi}(x)\Dx - f(0) \chi(t)\int\limits_{\Ck}\chi(x)\Dx\}\frac{1}{t}dt
\end{eqnarray}
Usng the fact that ,
\begin{eqnarray}\label{16}
\int\limits_{0}^{1}\zeta_{t^{-1}}(\hat{f},\check{\chi})\frac{1}{t}dt=\int\limits_{1}^{\infty}\zeta_{t}(\hat{f},\check{\chi})\frac{1}{t}dt\hspace{20pt}\mbox{ [Sustituting $t^{-1}$ for $t$] }
\end{eqnarray}
Using the result mentioned in \eqref{13}, we can say that, the integral, $\int\limits_{1}^{\infty}\zeta_{t}(\hat{f},\check{\chi})\frac{1}{t}dt$ is \textit{normally convergent} for all $s\in \C$, we assert that, the integral, $\int\limits_{0}^{1}\zeta_{t^{-1}}(\hat{f},\check{\chi})\frac{1}{t}dt$ also converges normally for every $s\in \C$.\\\\
Hence, we need to verify the convergence of only the error term $\mathcal{E}$ in order to conclude that, the \textit{Global Zeta Function} $\Zeta$ is normally convergent for every $s\in \C$.\\\\
Now, by definition of $\chi$ and $\check{\chi}$, they are orthogonal, by orthogonality relations, we assert that, 
\begin{center}
	$\int\limits_{\Ck}\chi(x)\Dx=0$,  and,  $\int\limits_{\Ck}\check{\chi}(x)\Dx=0$\hspace*{20pt}[Since, $\chi$ is non-trivial on $\Ik^{1}$]
\end{center}
Therefore, from \eqref{15}, we obtain, $\mathcal{E}=0$ for $\chi$ non-trivial on $\Ik^{1}$.\\\\
If, $\chi$ is trivial on $\Ik^{1}$, a priori, we have the representation,
\begin{center}
	$\chi=\mu|.|^{s}$
\end{center}
Then, we can write,
\begin{center}
	$\chi=|.|^{s'}$, where, $s'=s-i\tau$,  for some $\tau\in \R$\hspace*{20pt}[Since $\mu$ is unitary]
\end{center}
hence, evaluating $\mathcal{E}$ using above expression for $\chi$, we get,
\begin{center}
	$\mathcal{E}=\int\limits_{0}^{1}\{\hat{f}(0)t^{s'-1}Vol(\Ck)-f(0)t^{s'}Vol(\Ck) \}\frac{1}{t}dt$
\end{center}
\begin{center}\label{17}
	$=Vol(\Ck)\{\frac{\hat{f}(0)}{s'-1}-\frac{f(0)}{s'}\}$\hspace*{10pt}[Since, $\chi|_{\Ik^{1}}=1\Rightarrow \check{\chi}|_{\Ik^{1}}=1\Rightarrow\check{\chi}|_{\Ck}=1$]
\end{center}
Therefore, using \eqref{10} and \eqref{14}, we conclude that, $\Zeta$ is \textit{normally convergent} for every $s$, and since, $\mathcal{E}$ is a \textit{meromorphic function}, thus we obtain our desired \textit{meromorphic extension} of $\Zeta$ over $\C$
\item From above, we have, when $\chi$ is non-trivial on $\Ik^{1}$, then, $\mu\neq |.|^{-i\tau}$, $\tau\in \R$, then, $\mathcal{E}=0$. Hence, $\Zeta$ is \textit{Holomorphic} everywhere.\\\\
When, $\chi$ is trivial on $\Ik^{1}$, then, $\mu=|.|^{-i\tau}$, $\tau\in \R$. Hence,
\begin{center}
$\mathcal{E}=Vol(\Ck)\{\frac{\hat{f}(0)}{s'-1}-\frac{f(0)}{s'}\}$\hspace*{20pt}[ From \eqref{17}]
\end{center}
Hence $\Zeta$ is \textit{holomorphic} everywhere, except at the points, $s=i\tau$ and, $s=1+i\tau$, for $\tau\in \R$.\\\\
The respective residues can be evaluated as:
\begin{center}
	$- Vol(\Ck)f(0)$,  and,  $Vol(\Ck) \hat{f}(0)$
\end{center}
\begin{center}
	i.e.,  $	-\kappa f(0)$,  and,  $\kappa \hat{f}(0)$  
\end{center}
respectively, where, we have, $\kappa:=Vol.(\sC_{K}^{1})$.
\item Using the identities in  \eqref{10} and \eqref{14}, we get, for the \textit{Global Zeta Functions} $\Zeta$,
\begin{center}
	$	\Zeta=\int\limits_{0}^{1}\zeta_{t}(f,\chi)\frac{1}{t}dt+\int\limits_{0}^{1}\zeta_{t^{-1}}(\hat{f},\check{\chi})\frac{1}{t}dt+\mathcal{E}(f,\chi)$
\end{center}
\begin{eqnarray}\label{18}
	=\int\limits_{0}^{1}(\int\limits_{\Ik}f(tx)\chi(tx)\Dx)\frac{1}{t}dt+\int\limits_{0}^{1}(\int\limits_{\Ik}\hat{f}(tx)\check{\chi}(tx)\Dx)\frac{1}{t}dt+\mathcal{E}(f,\chi)
\end{eqnarray}
Moreover, applying propperties of \textit{Fourier Transform} for $f\in \sS(\Ak)$, we get,
\begin{eqnarray}\label{19}
	\hat{\hat{f}}=f(-x)\mbox{,  and,  }\check{\check{\chi}}=\chi
\end{eqnarray}
 Substituting $\hat{f}$ and $\check{\chi}$ instead of $f$ and $\chi$ in \eqref{18}, we get,
 \begin{center}
 	$	\Zetaa=\int\limits_{1}^{\infty}\zeta_{t}(\hat{f},\check{\chi})\frac{1}{t}dt+\int\limits_{1}^{\infty}\zeta_{t^{-1}}(\hat{\hat{f}},\chi)\frac{1}{t}dt+\mathcal{E}(\hat{f},\check{\chi})$
 \end{center}
\begin{eqnarray}\label{20}
=\int\limits_{1}^{\infty}(\int\limits_{\Ik}\hat{f}(tx)\check{\chi}(tx)\Dx)\frac{1}{t}dt+\int\limits_{1}^{\infty}(\int\limits_{\Ik}f(-tx)\chi(tx)\Dx)\frac{1}{t}dt+\mathcal{E}(\hat{f},\check{\chi})
\end{eqnarray}

But, from \eqref{15}, we have,
\begin{center}
	$\mathcal{E}(f,\chi)=\int\limits_{0}^{1}\{\hat{f}(0)\check{\chi}(t^{-1}) \int\limits_{\Ck} \check{\chi}(x)\Dx - f(0) \chi(t)\int\limits_{\Ck}\chi(x)\Dx\}\frac{1}{t}dt$
\end{center}
Hence,
\begin{center}
	$\mathcal{E}(\hat{f},\check{\chi})=\int\limits_{0}^{1}\{\hat{\hat{f}}(0)\check{\check{\chi}}(t^{-1}) \int\limits_{\Ck} \check{\check{\chi}}(x)\Dx - \hat{f}(0) \check{\chi(t)}\int\limits_{\Ck}\check{\chi(x)}\Dx\}\frac{1}{t}dt$
\end{center}
\begin{center}
	$=\int\limits_{0}^{1}\{\hat{f}(0)\check{\chi}(t^{-1}) \int\limits_{\Ck} \check{\chi}(x)\Dx - f(0) \chi(t)\int\limits_{\Ck}\chi(x)\Dx\}\frac{1}{t}dt$
\end{center}
[Using properties \eqref{19}]
  \begin{center}
  	$=\mathcal{E}(f,\chi)$
  \end{center}
Showing that, $\mathcal{E}$ is \textit{invariant} under the transformation map,
\begin{center}
	$(f,\chi)\mapsto(\hat{f},\check{\chi})$
\end{center}

Also, given the fact that, $\chi=\mu|.|^{s}$, and also $\chi$ being invariant under the tranformation map,
\begin{center}
	$tx\mapsto -tx$
\end{center}
Therefore, $\chi(tx)=\chi(-tx)$ for every $t\in \R$ and, $x\in \Ik$, since $\chi$ is an \textit{Id\`{e}le Class Character}.\\\\
Therefore, substituting $\chi(-tx)$ instead of $\chi(tx)$ in the second integral in the \textit{Equation\eqref{20}}, we obtain,
\begin{center}
	$\Zeta=\Zetaa$
\end{center}
Which establishes our desired result and completes the proof of the theorem.

\end{proof}

\end{enumerate}
	\end{proof}
	
\begin{rmk}\label{rmk15}
	In the \textit{Theorem\eqref{thm3}} the \textit{volume} of the \textit{Id\`{e}le Class Group of norm $1$}, $\Ck:=\Ik^{1}/K^{*}$ is measured with respect to the \textit{Haar Measure} on $\mathcal{C}_{K}$ defined by $\Dx$ and the counting measure on $K^{*}$.
\end{rmk}	
	\begin{rmk}\label{rmk16}
		By further calculation it can be deduced that,
		\begin{center}
			$Vol.(\Ck)=-Res_{s=1}\zeta_{K}(s)$
		\end{center}
	Where, $\zeta_{K}(s)$ denotes the \textit{Dedekind Zeta Fumction} on the global field $K$.
	\end{rmk}
\begin{rmk}
	If we consider the global field $K$ to be a \textit{funcion Field} instead of a \textit{number field}, the same statement of the \textit{main theorem} mentioned above holds true for the \textit{Global Zeta Functions} $\Zeta$ of $K$, although the proof differs significantly from that in case of the \textit{number fields}.
\end{rmk}
	
	\vspace{80pt}
\section*{Acknowledgments}
	I'll always be grateful to \textbf{Prof. Debargha Banerjee} ( Associate Professor, Department of Mathematics, IISER Pune, India ) for supervising my research project on this topic. His kind guidance helped me immensely in detailed understanding of this topic.

\end{document}